\documentclass{amsart}

\usepackage{graphicx,amssymb,mathrsfs,amsmath,color,fancyhdr}
\usepackage[all]{xy}
\newtheorem{theorem}{Theorem}[section]
\newtheorem{lemma}[theorem]{Lemma}

\theoremstyle{definition}

\newtheorem{proposition}[theorem]{Proposition}
\newtheorem{problem}{Problem}

\theoremstyle{remark}
\newtheorem{remark}[theorem]{Remark}

\numberwithin{equation}{section}

\UseRawInputEncoding

\newcommand{\bdot}{\boldsymbol{\cdot}}

\begin{document}

\title[A reciprocity on finite abelian groups]
{A reciprocity on finite abelian groups involving zero-sum sequences}

\begin{abstract}
In this paper, we present a reciprocity on finite abelian groups involving zero-sum sequences. Let $G$ and $H$ be finite abelian groups with $(|G|,|H|)=1$. For any positive integer $m$, let $\mathsf M(G,m)$ denote the set of all zero-sum sequences over $G$ of length $m$. We have the following reciprocity
$$|\mathsf M(G,|H|)|=|\mathsf M(H,|G|)|.$$
Moreover, we provide a combinatorial interpretation of the above reciprocity using ideas from rational Catalan combinatorics. We also present and explain some other symmetric relationships on finite abelian groups with methods from invariant theory. Among others, we partially answer a question proposed by Panyushev in a generalized version.
\end{abstract}

\author{Dongchun Han}
\address{Department of Mathematics, Southwest Jiaotong University, Chengdu 610000, P.R. China}
\email{handongchun@swjtu.edu.cn}
\author{Hanbin Zhang}
\address{School of Mathematics (Zhuhai), Sun Yat-sen University, Zhuhai 519082, Guangdong, P.R. China}
\email{zhanghb68@mail.sysu.edu.cn}


\maketitle

\section{Introduction}

Let $G$ be an additive finite abelian group. By a sequence over $G$, we mean a finite sequence of terms from $G$ which is unordered and repetition of terms is allowed (see Section 2 for more discussion). Let $S=g_1\bdot\ldots\bdot g_k$ be a sequence over $G$, where $g_1,\ldots,g_k\in G$ and $k$ is called the length of $S$. We define $\sigma(S)=g_1+\cdots+g_k$. We say that $S$ is a zero-sum sequence if $\sigma(S)$ equals 0, the identity of $G$. In fact, zero-sum theory, an active branch of combinatorial number theory, studies various interesting problems about zero-sum sequences, we refer to \cite{GG} for a survey on zero-sum theory. In this paper, we first present a reciprocity on finite abelian groups involving zero-sum sequences.

Let $|G|=n$ and $k,m$ be positive integers with $k\le n$. For any $g\in G$, we denote
$$\mathsf N(G,k,g)=\{S\text{ is a subset of }G\text{ }|\text{ }\sigma(S)=g\text{ and }|S|=k\}$$
and
$$\mathsf M(G,m,g)=\{S\text{ is a sequence over }G\text{ }|\text{ }\sigma(S)=g\text{ and }|S|=m\}.$$
In particular, we denote
$$\mathsf N(G,k):=\mathsf N(G,k,0)\text{ and }\mathsf M(G,k):=\mathsf M(G,k,0).$$

In 1975, using generating functions, Fredman \cite{Fred} proved the following very interesting reciprocity
\begin{equation}\label{Fredman}
|\mathsf M(\mathbb Z/n\mathbb Z,m)|=|\mathsf M(\mathbb Z/m\mathbb Z,n)|.
\end{equation}
More generally, he proved that $|\mathsf M(\mathbb Z/n\mathbb Z,m,i)|=|\mathsf M(\mathbb Z/m\mathbb Z,n,i)|,$
where
$$\mathsf M(\mathbb Z/n\mathbb Z,m,i)=\{S\text{ is a sequence over }\mathbb Z/n\mathbb Z\text{ }|\text{ }\sigma(S)\equiv i{\pmod n}\text{ and }|S|=m\}.$$
He also provided a combinatorial explanation of the above reciprocity using a necklace interpretation. In the following, for simplicity, let $C_n$ denote a
cyclic group with $n$ elements and $C_n^r$ the direct sum of $r$ copies of $C_n$.

Later in 1999, Elashvili, Jibladze and Pataraia \cite{EJ,EJP} rediscovered the same result, but they employed a classical idea from invariant theory introduced by Molien \cite{Mol} and a result of Almkvist and Fossum \cite{AlmF}. It was remarked in \cite[Introduction]{EJP} that N. Alon also independently proved (\ref{Fredman}). Meanwhile, G. Andrews, N. Alon and R. Stanley independently obtained the counting formula for $\mathsf M(C_n,m)$; see \cite[Introduction and Section 3]{EJP}. It is natural to ask whether there exists a similar reciprocity (i.e., $|\mathsf M(G,|H|)|=|\mathsf M(H,|G|)|$) for general finite abelian group. We shall show that (see Proposition \ref{GCp}), in general, the above reciprocity does not always hold for any two abelian groups. While, we can prove that it always holds for two abelian groups with $(|G|,|H|)=1$ and we also provide a combinatorial interpretation to explain this phenomenon. Recall that for any positive integers $n,m$ with $(n,m)=1$, the $(n,m)$-Catalan number $\mathsf {Cat}_{n,m}$ is defined as
$$\mathsf {Cat}_{n,m}:=\frac{1}{n+m}\binom{n+m}{n},$$
which is not only a natural generalization of the Catalan numbers $\mathsf {Cat}_{n}:=\mathsf {Cat}_{n,n+1}$ but also related to many problems in combinatorics, representation theory and geometry (see Section 2 for more discussion). A typical object counted by $\mathsf {Cat}_{n,m}$ is the set $\mathcal{D}_{n,m}$ of all $(n,m)$-Dyck paths which is defined as the number of lattice paths from $(0,0)$ to $(n,m)$ which only use unit steps $(1,0)$
or $(0,1)$ and stay above the diagonal line $y=\frac{m}{n}x$. The following theorem is our main result.

\begin{theorem}\label{mainth1}
Let $k,m,n$ be positive integers. Let $G$ and $H$ be two abelian groups. Then we have
\begin{enumerate}

\item If $(|G|,|H|)=1$, there are bijections between
      $$\mathsf M(G,|H|),\text{ }\mathsf M(H,|G|),\text{  and }\mathcal{D}_{|G|,|H|}.$$
      Therefore
      $$|\mathsf M(G,|H|)|=|\mathsf M(H,|G|)|=\mathsf {Cat}_{|G|,|H|}.$$

\item For abelian groups $G=C_n^r$ and $H=C_m^r$ with $(n,m^r)=(n^r,m)$, we have
      $$|\mathsf M(G,|H|)|=|\mathsf M(H,|G|)|.$$

\item For $k< |G|$ with $(k,|G|)=1$, there are bijections between
      $$\mathsf N(G,k),\text{ }\mathsf N(G,|G|-k),\text{ and }\mathcal{D}_{k,|G|-k}.$$
      Therefore
      $$|\mathsf N(G,k)|=|\mathsf N(G,|G|-k)|=\frac{1}{|G|}\binom{|G|}{k}=\mathsf {Cat}_{k,|G|-k}.$$
\end{enumerate}
\end{theorem}
Note that in Theorem \ref{mainth1}.(1), the equality $|\mathsf M(G,|H|)|=|\mathsf M(H,|G|)|$ only depends on the condition that $(|G|,|H|)=1$ (but not the group structures of $G$ and $H$). With the bijections in Theorem \ref{mainth1}.(1), the intuitive symmetry between $\mathcal{D}_{|G|,|H|}$ and $\mathcal{D}_{|H|,|G|}$ provides a combinatorial explanation of the above reciprocity. When $G$ is a cyclic group and $(|G|,m)=1$, while proving a conjecture of Armstrong \cite{AHJ} concerning the average size of the simultaneous core partitions, Johnson \cite[Lemma 27]{John} implicitly obtained a bijection between $\mathsf M(G,m)$ and $\mathcal{D}_{|G|,m}$ (see Remark \ref{Johnsonproof} for more discussion). In fact, the equalities in Theorem \ref{mainth1} can be observed from explicit counting formulas for $\mathsf N(G,k,g)$ and $\mathsf M(G,m,g)$ for arbitrary abelian group $G$ and any positive integers $k< |G|$ and $m$. Formula for $\mathsf N(G,k,g)$ (resp. $\mathsf M(G,m,g)$) was obtained in \cite{LW,Ko} (resp. \cite{MRW}), via sieve method (known as Li-Wan's sieve method) and generating functions. Moreover, in \cite{Pan}, Panyushev implicitly obtained these counting formulas which follow as consequences of some more general works in the study of invariant theory by Molien \cite{Mol} and by Almkvist and Fossum \cite{Alm,AlmF} (see Section 3 for more detailed explanation). When $r=1$, Theorem \ref{mainth1}.(2) simplifies to (\ref{Fredman}). Panyushev \cite[Theorem 4.2]{Pan} obtained a special case of Theorem \ref{mainth1}.(3) when $G$ is cyclic. For any positive integer $m$, let $v_2(m)$ be the 2-adic valuation of $m$. Our following result provides a characterization of the equality $|\mathsf N(G,k)|=|\mathsf N(G,|G|-k)|$.

\begin{theorem}\label{subset-reci}
Let $G\cong C_{n_1}\oplus\cdots\oplus C_{n_r}$ be a finite abelian group and $k$ a positive integer with $k\le |G|-1$. Then we have $|\mathsf N(G,k)|=|\mathsf N(G,|G|-k)|$ if and only if one of the following conditions holds:
\begin{enumerate}

\item $|G|$ is odd;

\item $r\ge2$ and $2|n_{r-1}$;

\item $v_2(k)<v_2(n_r)$.

\end{enumerate}

\end{theorem}

Next, we provide a generalization of Theorem \ref{mainth1} which is motivated by a result of Panyushev \cite[(3.5)]{Pan}. We briefly recall some definitions and notation. Let $G$ be a finite group and $V$ a $G$-module. Let $\big(\mathcal{S}(V)\otimes\wedge(V)\big)_{G,\chi}$ denote the isotypic component in the symmetric tensor exterior algebra of $V$ corresponding to an irreducible representation $\chi$. It is a bi-graded vector space and its Poincar\'{e} series is the formal power series
$$\mathcal{F}\Big(\big(\mathcal{S}(V)\otimes\wedge(V)\big)_{G,\chi};s,t\Big)=\sum_{p,m\ge 0}\dim \big(\mathcal{S}^p(V)\otimes\wedge^m(V)\big)_{G,\chi}s^pt^m.$$
Actually, counting formulas for $\mathsf N(G,k,g)$ and $\mathsf M(G,m,g)$ can be derived from this Poincar\'{e} series (see Section 3). We denote $\big(\mathcal{S}(V)\otimes\wedge(V)\big)_{G}=\big(\mathcal{S}(V)\otimes\wedge(V)\big)_{G,\chi_0}$ if $\chi_0$ is the trivial representation. Let $C_{q+m}$ and $C_{p+m}$ be two cyclic groups of order $q+m$ and $p+m$. Let $\mathcal{R}$ (resp. $\tilde{\mathcal{R}}$) be the regular representation of $C_{q+m}$ (resp. $C_{p+m}$) over $\mathbb C$, the field of complex numbers. Panyushev \cite{Pan} proved the following interesting generalization of (\ref{Fredman}):
\begin{equation}\label{Panyushev}
\dim\big(\mathcal{S}^p(\mathcal{R})\otimes\wedge^m(\mathcal{R})\big)_{C_{q+m}}=
\dim\big(\mathcal{S}^q(\tilde{\mathcal{R}})\otimes\wedge^m(\tilde{\mathcal{R}})\big)_{C_{p+m}}.
\end{equation}
Following the idea of Panyushev, we have the following generalization.
\begin{theorem}\label{mainth2}
Let $p,q,m$ be non-negative integers and let $G_{q+m}$ and $H_{p+m}$ be two abelian groups of orders $q+m$ and $p+m$. Let $\mathcal{R}$ (resp. $\tilde{\mathcal{R}}$) be the regular representation of $G_{q+m}$ (resp. $H_{p+m}$) over $\mathbb C$. If $(p,q,m)=1$, then we have
\begin{equation}\label{recige}
\dim\big(\mathcal{S}^p(\mathcal{R})\otimes\wedge^m(\mathcal{R})\big)_{G_{q+m}}=
\dim\big(\mathcal{S}^q(\tilde{\mathcal{R}})\otimes\wedge^m(\tilde{\mathcal{R}})\big)_{H_{p+m}}.
\end{equation}
Indeed, both dimensions are equal to $\frac{1}{p+q+m}\binom{p+q+m}{p,q,m}$.
\end{theorem}

Note that, when $m=0$, (\ref{recige}) simplifies to the reciprocity in Theorem \ref{mainth1}.(1). Panyushev asked for a combinatorial interpretation of (\ref{Panyushev}) (see the discussion after Theorem 3.2 in \cite{Pan}). We shall provide one for the above generalized theorem in a special case using a necklace interpretation.

The following sections are organized as follows. In Section 2, we shall introduce some preliminary results and notation. In Section 3, we discuss the counting formulas for $\mathsf N(G,k,g)$ and $\mathsf M(G,m,g)$ as well as some related results in invariant theory. In Section 4, we prove Theorems \ref{mainth1}, \ref{subset-reci}, and \ref{mainth2}. Finally, in Section 5, we provide some concluding remarks and propose a problem for future study.

\section{Preliminaries}

In this section, we will provide more rigorous definitions and notation. We also introduce some preliminary results that will be used repeatedly below.

Let $\mathbb C$ be the field of complex numbers. Denote by $\mathbb N$ the set of positive integers and let $\mathbb N_0=\mathbb N\cup\{0\}$. Let $G$ be a finite abelian group written additively. By the fundamental theorem of finite abelian groups we have
$$G\cong C_{n_1}\oplus\cdots\oplus C_{n_r}=\langle e_1\rangle\oplus\cdots\oplus \langle e_r\rangle$$
where $r=\mathsf r(G)\in \mathbb{N}_0$ is the rank of $G$, $1<n_1|\cdots|n_r\in\mathbb{N}$ are positive integers. Moreover, $n_1,\ldots,n_r$ are uniquely determined by $G$. We also use $C_n^r$ to denote the abelian group of the following form
$$\overbrace{C_n\oplus\cdots\oplus C_n}^r.$$
In combinatorial number theory, a $sequence$ over $G$ is defined to be an element of the multiplicatively written free abelian monoid $\big(\mathcal F(G),\bdot\big)$; see Chapter 5 of \cite{GH} for detailed explanation and see \cite{GeG} for a discussion of the notation. In particular, we define:
$$ g^{[k]}=\overset{k}{\overbrace{g\bdot\ldots\bdot g}}\in \mathcal F (G),$$
for $g \in G$ and $k \in \mathbb N_0$.

For a sequence $S=g_1\bdot\ldots\bdot g_l\in \mathcal F (G)$ we call
\begin{itemize}

\item $|S|=l\in \mathbb{N}_0$ the $length$ of $S$;

\item $\sigma(S)=g_1+\cdots+g_l\in G$ the $sum$ of $S$;

\item $S$ a $zero$-$sum$ $sequence$ if $\sigma(S)=0$.
\end{itemize}

For the convenience of our bijective proofs later, we provide the following modified notation of sequences.
\begin{remark}\label{important}
We write a sequence $S$ over $C_n=\langle e\rangle=\{0,e,2e,\ldots,(n-1)e\}$ as a vector $(x_0,\ldots,x_{n-1})\in\mathbb N^n$, where $x_i$ is the multiplicity that $ie$ occurs in $S$, that is, in the previous notation
$$S=0^{[x_0]}\bdot e^{[x_1]}\bdot\ldots\bdot \big((n-1)e\big)^{[x_{n-1}]}.$$
Generally, let $G$ be a finite abelian group of order $n$ and
$$G\cong C_{n_1}\oplus\cdots\oplus C_{n_r}=\langle e_1\rangle\oplus\cdots\oplus \langle e_r\rangle$$
where $n_1|\cdots|n_r\in\mathbb{N}$ are positive integers and $n_1\cdots n_r=n$. Similar to the above case, every element in $G$ can be written uniquely as $$a_1e_1+\cdots+a_re_r$$ for some positive $a_i\le n_i-1$ and $1\le i\le r$. In order to attach a similar vector $(y_0,\ldots,y_{n-1})\in\mathbb N^n$ to an arbitrary sequence over $G$, we introduce the following labels for elements in $G$. Any integer $k\in [0,n_1\cdots n_r-1]$ can be written uniquely in the following form
$$k=b_{r}n_1\cdots n_{r-1}+b_{r-1}n_1\cdots n_{r-2}+\cdots+b_2n_1+b_1$$
where $b_i\in[0,n_i-1]$ for $1\le i\le r$. Therefore, let $g=a_1e_1+\cdots+a_re_r$, then $g$ will be attached with the label
$$l_g=a_{r}n_1\cdots n_{r-1}+\cdots+a_2n_1+a_1.$$ With this label, a vector $(y_0,\ldots,y_{n-1})\in\mathbb N^n$ corresponds to a sequence $S$ over $G$, where $y_i$ is the multiplicity that $g=a_1e_1+\cdots+a_re_r$ (with $l_g=i$) occurs in $S$. In the previous notation
\begin{equation}\label{labelG}
S=\prod_{g\in G}g^{[y_{l_g}]}.
\end{equation}
\end{remark}

\medskip

For any positive integers $n,m$ with $(n,m)=1$, the {\sl $(n,m)$-Catalan number} is defined as
$$\mathsf {Cat}_{n,m}=\frac{1}{n+m}\binom{n+m}{n}$$
which is a natural generalization (take $m=n+1$) of Catalan numbers $\mathsf {Cat}_{n}:=\frac{1}{2n+1}\binom{2n+1}{n}$ (see \cite{Stan1}). For all relatively prime pair $(n,m)$, these $(n,m)$-Catalan numbers are also called the rational Catalan numbers. A typical object counted by $\mathsf {Cat}_{n,m}$ is the set $\mathcal{D}_{n,m}$ of all $(n,m)$-Dyck paths which is defined as the number of lattice paths from $(0,0)$ to $(n,m)$ which only use unit steps $(1,0)$
or $(0,1)$ and stay above the diagonal line $y=\frac{m}{n}x$ (see \cite{Biz}). In fact, the rational Catalan numbers (and their $q$- or $(q,t)$- analogs) arose naturally in many research areas, such as simultaneously core partitions, non-crossing partitions, parking functions, Hecke algebra, affine Springer varieties, compactified Jacobians of singular curves, etc.; see, e.g., \cite{LS,GarH,Arm,ARW,GM,AHJ,GM2,ALW,BR,GMV}. In fact, studies related to the rational Catalan numbers are called rational Catalan combinatorics, which is currently an active branch of combinatorics. Moreover, the rational Catalan numbers have an interesting and deep algebraic generalization (see \cite{Hai1}).

\section{Invariants of finite abelian groups}

In this section, we first recall the following counting formulas, which are explicitly obtained in \cite{LW,Ko,MRW} via sieve method (known as Li-Wan's sieve method) and generating functions.

\begin{theorem}\label{LiW}{\rm (\cite{LW,Ko,MRW})}
Let $G=C_{n_1}\oplus\cdots\oplus C_{n_r}=\langle e_1\rangle\oplus\cdots\oplus \langle e_r\rangle$ be a finite abelian group with $|G|=n=n_1\cdots n_r$ and $n_1|\cdots|n_r$. Let $g=g_1e_1+\cdots+g_re_r\in G$, where $g_i\in [0,n_i-1]$ for every $i$. Then for any positive integers $k<n$ and $m$, we have
\begin{enumerate}

\item $|\mathsf N(G,k,g)|=\frac{1}{n}\sum_{d|(n,k)}\Phi_G(g,d)(-1)^{k+\frac{k}{d}}
    \binom{n/d}{k/d}$,

\item $|\mathsf M(G,m,g)|=\frac{1}{n+m}\sum_{d|(n,m)}\Phi_G(g,d)\binom{n/d+m/d}{n/d}$,

\end{enumerate}
where $\Phi_G(g,d)=\sum_{\chi\in\widehat{G},\text{ord}(\chi)=d}\chi(g)=\sum_{l|d,(n_i,l)|g_i}\mu(\frac{d}{l})\prod_{i=1}^r(n_i,l)$.
In particular, for any positive integers $k< n$ and $m$, we have
$$|\mathsf N(G,k)|=\frac{1}{n}\sum_{d|(n,k)}\varphi_G(d)(-1)^{k+\frac{k}{d}}\binom{n/d}{k/d}$$
and
$$|\mathsf M(G,m)|=\frac{1}{n+m}\sum_{d|(n,m)}\varphi_G(d)\binom{n/d+m/d}{n/d},$$
where $\varphi_G(d)=\sum_{l|d}\mu(d/l)\prod_{i=1}^r(n_i,l)$.
\end{theorem}

It is easy to see that the equalities in Theorem \ref{mainth1}.(1) and (3) follow from the above theorem. As we have mentioned before, employing ideas of Molien \cite{Mol} and results of Almkvist and Fossum \cite{Alm,AlmF} from invariant theory, Panyushev \cite{Pan} implicitly obtained a proof of Theorem \ref{LiW} from the perspective of invariant theory. As Theorem \ref{mainth2} is motivated by Panyushev's result, we briefly recall his approach for the convenience of readers.

Let $G$ be a finite group and $V$ a finite dimensional representation of $G$ over $\mathbb C$. Let $\big(\mathcal{S}(V)\otimes\wedge(V)\big)_{G,\chi}$ denote the isotypic component in symmetric tensor exterior algebra of $V$ corresponding to an irreducible representation $\chi$. It is a bi-graded vector space and its Poincar\'{e} series is the formal power series
\begin{equation}\label{Almkvist}
\mathcal{F}\Big(\big(\mathcal{S}(V)\otimes\wedge(V)\big)_{G,\chi};s,t\Big)=\sum_{p,m\ge 0}\dim \big(\mathcal{S}^p(V)\otimes\wedge^m(V)\big)_{G,\chi}s^pt^m.
\end{equation}
Based on a remarkable theorem of Molien (\cite{Mol}, \cite[Section 2]{Stan}), Almkvist (\cite[Theorem 1.33]{Alm}) proved the following formula
\begin{equation}\label{MoAF}
\mathcal{F}\Big(\big(\mathcal{S}(V)\otimes\wedge(V)\big)_{G,\chi};s,t\Big)=\frac{1}{|G|}\sum_{g\in G}\frac{\det_V(E+gt)}{\det_V(E-gs)},
\end{equation}
where $E$ is the identity matrix in $GL(V)$. Later, Panyushev obtained the following easy consequence of (\ref{MoAF}).
\begin{lemma}{\rm (\cite[Lemma 3.1]{Pan})}\label{Panlemma}
Let $G$ be a finite group and $\mathcal{R}$ the regular representation of $G$ over $\mathbb C$, then we have
$$\mathcal{F}\Big(\big(\mathcal{S}(\mathcal{R})\otimes\wedge(\mathcal{R})\big)_{G,\chi};s,t\Big)
=\frac{\deg(\chi)}{|G|}\sum_{d\ge1}\Big(\sum_{g,\text{ord}(g)=d}\text{tr}(\chi(g^{-1}))\cdot
\Big(\frac{1-(-t)^d}{1-s^d}\Big)^{|G|/d}\Big).$$
In particular, if $G$ is abelian, then we have
\begin{equation}\label{abel}
\mathcal{F}\Big(\big(\mathcal{S}(\mathcal{R})\otimes\wedge(\mathcal{R})\big)_{G,\chi};s,t\Big)
=\frac{1}{|G|}\sum_{d\ge1}\Big(\sum_{g,\text{ord}(g)=d}\chi(g^{-1})\cdot
\Big(\frac{1-(-t)^d}{1-s^d}\Big)^{|G|/d}\Big)
\end{equation}
and
\begin{equation}\label{abel0}
\mathcal{F}\Big(\big(\mathcal{S}(\mathcal{R})\otimes\wedge(\mathcal{R})\big)_{G};s,t\Big)
=\frac{1}{|G|}\sum_{d\ge1}\Big(\varphi_G(d)\Big(\frac{1-(-t)^d}{1-s^d}\Big)^{|G|/d}\Big),
\end{equation}
where $\varphi_G(d)$ is the number of elements in $G$ of order $d$.
\end{lemma}
The special cases when $t=0$ in (\ref{abel0}) was obtained in 1978 by Almkvist and Fossum \cite[V.1.8]{AlmF}, when $s=0$ we refer to \cite[Section 4]{Pan}.

Now we provide a detailed discussion of the above result in the case when $G$ is abelian and show that (\ref{abel}) actually provides the counting formulas for $\mathsf N(G,k,g)$ and $\mathsf M(G,m,g)$ simultaneously. Let $G$ be a finite abelian group of order $n$. Let $V=\langle e_1\rangle\oplus\cdots\oplus\langle e_n\rangle$ be the regular representation of $G$ over $\mathbb C$ with $g\cdot e_i:=\chi_i(g)e_i$, where $g\in G$ and $\chi_i\in\widehat{G}$. For any $b\in G$, let $\widehat{b}$ be the element in $\widehat{G}$ which corresponds to $b$ under the isomorphism $G\cong\widehat{G}$. Let $\mathcal{S}(V)$ be the symmetric algebra of $V$. Note that $\mathcal{S}(V)=\bigoplus_{m\ge 0}\mathcal{S}^m(V)$ is a graded algebra, where $\mathcal{S}^m(V)$ is the homogeneous component of $\mathcal{S}(V)$ of degree $m$. The action of $G$ on $V$ can be naturally induced on $\mathcal S(V)$. For example, for any $e_{i_1}\otimes\cdots\otimes e_{i_k}\in \mathcal{S}(V)$, we have:
\begin{align*}
&g\cdot (e_{i_1}\otimes\cdots
\otimes e_{i_k}):=g\cdot e_{i_1}\otimes\cdots
\otimes g\cdot e_{i_k}\\
&=\chi_{i_1}(g)\cdot e_{i_1}\otimes\cdots
\otimes \chi_{i_k}(g)\cdot e_{i_k}=\chi_{i_1}\cdots\chi_{i_k}(g) (e_{i_1}\otimes\cdots
\otimes e_{i_k}).
\end{align*}
Let $\mathcal S^m(V)_{G,\widehat{b}}$ be the isotypic component in $\mathcal S^m(V)$ corresponding to $\widehat{b}\in\widehat{G}$, which is, by definition, spanned by all $e_{i_1}\otimes\cdots\otimes e_{i_m}\in \mathcal{S}^m(V)$ such that
\begin{align*}
&g\cdot (e_{i_1}\otimes\cdots
\otimes e_{i_m})=\chi_{i_1}\cdots\chi_{i_m}(g) (e_{i_1}\otimes\cdots
\otimes e_{i_m})\\
&=\widehat{b}(g) (e_{i_1}\otimes\cdots
\otimes e_{i_m})
\end{align*}
holds for any $g\in G$. Therefore, the above element $e_{i_1}\otimes\cdots\otimes e_{i_m}$ corresponds to, via the isomorphism $G\cong \widehat{G}$, a sequence $S=g_{i_1}\bdot\ldots\bdot g_{i_m}$ over $G$ with $\sigma(S)=b$, where $\widehat {g_{i_j}}=\chi_{i_j}$ for $1\le j\le m$. Consequently, we have
$$\dim\left(\mathcal S^m(V)_{G,\widehat{b}}\right)=|\mathsf M(G,m,b)|.$$
Similarly, by the definition of $\wedge^k(V)_{G,\widehat{b}}$, we have
$$\dim\left(\wedge^k(V)_{G,\widehat{b}}\right)=|\mathsf N(G,k,b)|.$$
Therefore, the Poincar\'e series $\mathcal{F}\Big(\big(\mathcal{S}(V)\otimes\wedge(V)\big)_{G,\widehat{g}};s,t\Big)$ provides the counting formulas for $\mathsf N(G,k,g)$ and $\mathsf M(G,m,g)$ simultaneously.

To obtain the precise formulas for $\mathsf N(G,k,g)$ and $\mathsf M(G,m,g)$, with Lemma \ref{Panlemma}, the last minor step is just an explicit calculation of $\sum_{g,\text{ord}(g)=d}\chi(g^{-1})$, or equivalently $\sum_{\chi,\text{ord}(\chi)=d}\chi^{-1}(g).$
Recall that
$$G=C_{n_1}\oplus\cdots\oplus C_{n_r}=\langle e_1\rangle\oplus\cdots\oplus \langle e_r\rangle$$
and $|G|=n$. Let
$$g=g_1e_1+\cdots+g_re_r\in G,$$
where $g_i\in [0,n_i-1]$ for every $i$. By basic representation theory and the principle of inclusion-exclusion (for the details we refer to \cite{Ko}), one obtains the following
\begin{equation}\label{chisumg}
\sum_{\chi,\text{ord}(\chi)=d}\chi^{-1}(g)=\sum_{l|d,(n_i,l)|g_i}\mu(\frac{d}{l})\prod_{i=1}^r(n_i,l).
\end{equation}
In particular,
\begin{equation}\label{varGd}
\varphi_G(d)=\sum_{l|d}\mu(\frac{d}{l})\prod_{i=1}^r(n_i,l).
\end{equation}
With (\ref{abel})-(\ref{varGd}), extracting the coefficient of $t^ks^m$, Theorem \ref{LiW} follows.

\begin{remark}
Note that, in order to prove Theorem \ref{LiW}.(1), Li and Wan \cite{LW} introduced a new sieve method which is different from the above method. Li-Wan's sieve method is useful to study counting problems in this flavor, we refer to their subsequent papers for detailed discussion \cite{LW08,LW2}. Kosters \cite{Ko} started from the following expansion ($G$ is considered as a multiplicative group here)
$$\sum_{i=0}^n\sum_{g\in G}N(G,i,g)gX^i=\prod_{h\in G}(1+hX)\in \mathbb C[G][X].$$ Then he used character theory to extract $N(G,i,g)$. From the perspective of generating functions, the proof of Kosters and the above proof of Panyushev (motivated by ideas of Molien, Almkvist and Fossum) which employed the Poincar\'{e} series are similar, but the details in these proofs can be different. There are several different and interesting ways to prove (\ref{MoAF}) and (\ref{abel}). For example, Molien's original idea (see \cite{Mol} and \cite[Section 2]{Stan}); Almkvist's method (see \cite[Theorem 1.33]{Alm}) and (\cite[Example 4.6]{Alm73}); Panyushev's approach (\cite[Lemma 3.1]{Pan}). For more detailed explanations of Molien's seminal result, we refer to \cite[Chapter 3]{NeuSm}.
\end{remark}

\begin{remark}
As we have mentioned that (\ref{MoAF}) is a more general formula, we briefly introduce its applications in more general settings (for non-abelian groups). Actually, the Poincar\'{e} series is a widely used and well-studied tool in invariant theory. For example, it was used to provide tight lower bounds for the Noether numbers of the quaternion group $Q_8$ and alternating group $A_4$ (see \cite{BS}). Here the Noether number of a finite group $G$, an interesting research topic in the interplay between invariant theory and zero-sum theory, is defined to be the maximal degree bound for the generators of the algebra of polynomial invariants of $G$. We refer to \cite[Chapter 5]{CDG} for a survey of studies of the Noether number, also to \cite{CDS,HanZh} for some recent results on the connection between zero-sum theory and the Noether number.
\end{remark}

\section{Proofs of our main results}

In this section, we first prove Theorem \ref{mainth1}, which can be regarded as a combinatorial interpretation of Theorem \ref{LiW} in some special cases.

Before we present the bijective proof, we first explain its basic idea by providing a direct and quick proof of the equality in Theorem \ref{mainth1}.(1). Let $G$ and $H$ be any two finite abelian groups with $|G|=n$, $|H|=m$, and $(n,m)=1$. Since $(n,m)=1$, it is easy to see that $mG=G$. Let $S=g_1\bdot\ldots\bdot g_m$ be any sequence over $G$ of length $m$. For any $g\in G$, consider the shifted sequence $g+S:=g+g_1\bdot\ldots\bdot g+g_m$, we have $\sigma(g+S)=mg+\sigma(S)$. Moreover, we have $\{\sigma(g+S)\ |\ g\in G\}=G$, or equivalently, there are exactly $\frac{1}{n}$ of all sequences of length $m$ over $G$ are zero-sum sequences. As there are $\binom{n+m-1}{m}$ sequences over $G$ of length $m$, we have $|\mathsf M(G,|H|)|=\frac{1}{n}\binom{n+m-1}{m}$ and the desired result follows. In the above proof, the idea of the shift is important and will also be the main strategy in the following bijective proof.\footnote{We thank the referee for suggesting to add this paragraph.}

With the help of Remark \ref{important}, we provide a necessary and sufficient condition for a sequence to be zero-sum. Recall that in the case when $G=C_n$ is the cyclic group of order $n$, a sequence $T=(y_0,\ldots,y_{n-1})$ over $G$ satisfies $\sigma(T)=0$ if and only if $\sum_{i=0}^{n-1}iy_i\equiv 0\pmod{n}.$ The following lemma is just a generalization of this fact and it is crucial in our following proof.
\begin{lemma}\label{zero-sum}
Let $G=C_{n_1}\oplus\cdots\oplus C_{n_r}=\langle e_1\rangle\oplus\cdots\oplus \langle e_r\rangle$ be a finite abelian group with $|G|=n=n_1\cdots n_r$ and $n_1|\cdots|n_r$. A sequence $T=(y_0,\ldots,y_{n-1})$ over $G$ satisfies $\sigma(T)=0$ if and only if the following congruences
\begin{equation}\label{gsimu1}
\sum_{k=0}^{n_1-1} k\Big(\sum_{0\le i_j\le n_j-1, j\in[2,r]}y_{i_rn_1\cdots n_{r-1}+\cdots+i_tn_1\cdots n_{t-1}+\cdots+i_2n_1+k}\Big)\equiv 0\pmod{n_1}
\end{equation}
and
\begin{equation}\label{gsimu}
\sum_{k=0}^{n_t-1} k\Big(\sum_{0\le i_j\le n_j-1, j\in[1,r]\setminus\{t\}}y_{i_rn_1\cdots n_{r-1}+\cdots+kn_1\cdots n_{t-1}+\cdots+i_2n_1+i_1}\Big)\equiv 0\pmod{n_t}
\end{equation}
hold simultaneously, where $t\in[2,r]$.
\end{lemma}
\begin{proof}
We first prove the case when $r=2$. For any $g=a_1e_1+a_2e_2\in G$, by Remark \ref{important}, $g$ is attached with the label
$$l_g=a_2n_1+a_1.$$
Assume that $\sigma(T)=c_1e_1+c_2e_2$. Therefore, $\sigma(T)=0$ if and only if $c_i\equiv 0\pmod{n_i}$ for $i=1,2$. We first calculate $c_1\pmod{n_1}$. Based on the above labels, for any $j\in[0,n_1-1]$ and $k\in[0,n_2-1]$, $\sigma\left((je_1+ke_2)^{[y_{kn_1+j}]}\right)=jy_{kn_1+j}e_1+ky_{kn_1+j}e_2$. Therefore, we have
$$c_1\equiv \sum_{j=0}^{n_1-1}j(y_j+y_{n_1+j}+\cdots+y_{(n_2-1)n_1+j})\pmod{n_1}$$
and
$$c_2\equiv \sum_{k=0}^{n_2-1}k(y_{kn_1}+y_{kn_1+1}+\cdots+y_{kn_1+n_1-1})\pmod{n_2}.$$
For the general case, let $g=a_1e_1+\cdots+a_re_r\in G$, by Remark \ref{important}, $g$ is attached with the label
$$l_g=a_{r}n_1\cdots n_{r-1}+\cdots+a_2n_1+a_1.$$
Assume that $\sigma(T)=c_1e_1+\cdots+c_re_r$. Therefore, $\sigma(T)=0$ if and only if $c_i\equiv 0\pmod{n_i}$ for $i\in[1,r]$. Moreover, for any $a_i\in[0,n_i-1]$ and $i\in[1,r]$, we have
$$\sigma\left((a_1e_1+\cdots+a_re_r)^{[y_{a_{r}n_1\cdots n_{r-1}+\cdots+a_2n_1+a_1}]}\right)
=\sum_{i=1}^ra_iy_{a_{r}n_1\cdots n_{r-1}+\cdots+a_2n_1+a_1}e_i.$$
Similar to the above, extracting $c_i$ ($1\le i\le r$), we get the desired result.
\end{proof}

Now, we are ready to prove Theorem \ref{mainth1}.

\medskip

{\sl Proof of the Theorem \ref{mainth1}.} (1) Let $|G|=n$ and $|H|=m$. Let $S=(x_0,\ldots,x_{n-1})$ be a zero-sum sequence over $G$ of length $m$, that is, $\sum_{i=0}^{n-1}x_i=m$. We shall construct a unique $(n,m)$-Dyck path $P$ which corresponds to the vector $(x_0,\ldots,x_{n-1})$. For any $i\in[0,n-1]$ and positive integer $t$ with $t\equiv i\pmod{n}$, we define $x_t=x_i$. The method used here is essentially the same as the proof of Theorem 12.1 in \cite{Loe}. The difference is that, in \cite{Loe}, each path was associated with a vector of length $n+m+1$ instead. We provide the complete proof here for the convenience of readers.

Firstly we construct a path (not necessarily a Dyck path) $Q$ from $(0,0)$ to $(n,m)$ as follows. Let $(i,\sum_{j=0}^{i-1}x_j)$ be the lowest lattice point in the $i$-th column of the path $Q$ where $i\in[1,n-1]$. We associate $Q$ with a vector $(y_1,\ldots,y_{n-1})\in\mathbb Z^{n-1}$, where $y_i=\sum_{j=0}^{i-1}x_j-\frac{m}{n}i$ for any $i\in[1,n-1]$. Note that for any $i\in[1,n-1]$ and positive integer $t$ with $t\equiv i\pmod{n}$, we have $y_t=y_i$.
Then it is obvious that $Q$ is an $(n,m)$-Dyck path if and only if
$$y_i>0\text{ holds for all }i\in[1,n-1].$$
As $(n,m)=1$ and $\sum_{j=0}^{i-1}x_j\le m$ for $i\in[1,n-1]$, it can be verified that $y_i\neq y_j$ for any $i\neq j$. Therefore we may denote the unique minimal element as $y_{\lambda}$ for some $\lambda\in[1,n-1]$. Moreover, we define $Q^l$ as a path obtained from the shifted vector
$$(x_l,x_{l+1},\ldots,x_{n-1},x_0,\ldots,x_{l-1}),$$
that is $(i,\sum_{j=0}^{i-1}x_{l+j})$ be the lowest lattice point in the $i$-th column of the path $Q^l$ where $i\in[1,n-1]$. Therefore we have $Q^{n}=Q$. We denote $y_i^l=\sum_{j=0}^{i-1}x_{l+j}-\frac{m}{n}i$. Similarly, $Q^l$ is a Dyck path if and only if
$$y_i^l>0\text{ holds for any }i\in[1,n-1].$$
Note that $y_i^l=y_{i+l}-y_l$. Now we consider the path $Q^{\lambda}$ and therefore $y_i^{\lambda}=y_{i+\lambda}-y_{\lambda}$. By the minimality of $y_{\lambda}$, we have $y_i^{\lambda}>0$ for any $i\in[1,n-1]$. Therefore we obtain a unique $(n,m)$-Dyck path $Q^{\lambda}$ which corresponds to the vector $(x_0,\ldots,x_{n-1})$.

\medskip

Conversely, let $P$ be an $(n,m)$-Dyck path from $(0,0)$ to $(n,m)$, then it is clear that $P$ corresponds to a vector $P=(x_0,\ldots,x_{n-1})$ with $(i,\sum_{j=0}^ix_j)$ be the highest lattice point in the $i$-th column of the path $P$. We assume that
$$\sum_{i=0}^{n-1}x_ii\equiv \alpha\pmod{n}.$$ As $\sum_{i=0}^{n-1}x_i=m$ and $(n,m)=1$, it is easy to see that there is a unique cyclic shift $(x_0',\ldots,x_{n-1}')$ of $(x_0,\ldots,x_{n-1})$ such that $\sum_{i=0}^{n-1}x_i'i\equiv 0\pmod{n}$.
By the first part of the proof, we know that $P$ is the unique Dyck path among other non-trivial cyclic shifts of $P$. Therefore, without loss of generality, we may assume that $P$ satisfies
\begin{equation}\label{cyclic0}
\sum_{i=0}^{n-1}x_ii\equiv 0\pmod{n}.
\end{equation}
As we have mentioned (see Remark \ref{labelG}) that $(x_0,\ldots,x_{n-1})$ corresponds to a sequence $T$ over $G$, though we do not necessarily have $\sigma(T)=0$. We will prove that there is a unique cyclic shift of $(x_0,\ldots,x_{n-1})$ which corresponds to a zero-sum sequence $T'$ over $G$.

By (\ref{cyclic0}) and the fact that $n_1|n$, it is easy to see that
\begin{equation}\label{gcyc2}
\sum_{j=0}^{n_1-1}j\Big(\sum_{0\le i_j\le n_j-1, j\in[1,r]\setminus\{1\}} x_{i_rn_1\cdots n_{r-1}+\cdots+i_2n_1+j}\Big)\equiv 0\pmod{n_1}.
\end{equation}
Therefore (\ref{gsimu1}) holds immediately. At the same time, according to (\ref{gsimu}), we may assume that
\begin{equation}\label{gn21}
\sum_{k=0}^{n_2-1} k\Big(\sum_{0\le i_j\le n_j-1, j\in[1,r]\setminus\{2\}}x_{i_rn_1\cdots n_{r-1}+\cdots+kn_1+i_1}\Big)\equiv \alpha_2\pmod{n_2}.
\end{equation}
Let $T^{2,l}$ be a sequence over $G$ obtained by cyclically shifting $T$ in the following way:
$$T^{2,l}=(x_l,\ldots,x_{n-1},x_0,x_1,\ldots,x_{l-1}).$$
We consider the sequence $T^{2,ln_1}$, then (\ref{gn21}) becomes
\begin{equation}\label{gn22}
\begin{aligned}
&\sum_{k=0}^{n_2-1} k\Big(\sum_{0\le i_j\le n_j-1, j\in[1,r]\setminus\{2\}}x_{i_rn_1\cdots n_{r-1}+\cdots+kn_1+i_1+ln_1}\Big)\\
&\equiv \sum_{k=0}^{n_2-1} (k+l-l)\Big(\sum_{0\le i_j\le n_j-1, j\in[1,r]\setminus\{2\}}x_{i_rn_1\cdots n_{r-1}+\cdots+(k+l)n_1+i_1}\Big)\\
&\equiv \alpha_2-lm\pmod{n_2}.
\end{aligned}
\end{equation}
As $(n_2,m)=1$, there exists a unique $l_2\in[0,n_2-1]$ such that $\alpha_2-l_2m\equiv 0\pmod{n_2}$. Therefore, the sequence $T^{2,l_2n_1}$ satisfies the case for $n_2$ in (\ref{gsimu}). Meanwhile, it is easy but necessary to see that
\begin{equation}\label{gn12}
\begin{aligned}
&\sum_{j=0}^{n_1-1}j\Big(\sum_{0\le i_j\le n_j-1, j\in[1,r]\setminus\{1\}} x_{i_rn_1\cdots n_{r-1}+\cdots+i_2n_1+j+l_2n_1}\Big)\\
&\equiv \sum_{j=0}^{n_1-1}(j+l_2n_1-l_2n_1)\Big(\sum_{0\le i_j\le n_j-1, j\in[1,r]\setminus\{1\}}  x_{i_rn_1\cdots n_{r-1}+\cdots+i_2n_1+j+l_2n_1}\Big)\\
&\equiv \sum_{j=0}^{n_1-1}(j+l_2n_1)\Big(\sum_{0\le i_j\le n_j-1, j\in[1,r]\setminus\{1\}}  x_{i_rn_1\cdots n_{r-1}+\cdots+i_2n_1+j+l_2n_1}\Big)-l_2n_1m\\
&\equiv 0-l_2n_1m\equiv 0\pmod{n_1},
\end{aligned}
\end{equation}
which means that (\ref{gsimu1}) still holds for the sequence $T^{2,l_2n_1}$. For simplicity, we still denote $T^{2,l_2n_1}$ by $(x_0,\ldots,x_{n-1})$. Next, we assume that
\begin{equation}\label{gn31}
\sum_{k=0}^{n_3-1} k\Big(\sum_{0\le i_j\le n_j-1, j\in[1,r]\setminus\{3\}}x_{i_rn_1\cdots n_{r-1}+\cdots+kn_1n_2+i_2n_1+i_1}\Big)\equiv \alpha_3\pmod{n_3}.
\end{equation}
Let $T^{3,l}$ be a sequence over $G$ obtained by cyclically shifting $T^{2,l_2n_1}$ in the following way:
$$T^{3,l}=(x_l,\ldots,x_{n-1},x_0,x_1,\ldots,x_{l-1}).$$
Then for the sequence $T^{3,ln_1n_2}$, (\ref{gn31}) becomes
\begin{equation}\label{gn32}
\begin{aligned}
&\sum_{k=0}^{n_3-1} k\Big(\sum_{0\le i_j\le n_j-1, j\in[1,r]\setminus\{3\}}x_{i_rn_1\cdots n_{r-1}+\cdots+kn_1n_2+i_2n_1+i_1+ln_1n_2}\Big)\\
&\equiv \sum_{k=0}^{n_3-1} (k+l-l)\Big(\sum_{0\le i_j\le n_j-1, j\in[1,r]\setminus\{3\}}x_{i_rn_1\cdots n_{r-1}+\cdots+(k+l)n_1n_2+i_2n_1+i_1}\Big)\\
&\equiv \alpha_3-lm\pmod{n_3}.
\end{aligned}
\end{equation}
As $(n_3,m)=1$, there exists a unique $l_3\in[0,n_3-1]$ such that $\alpha_3-l_3m\equiv 0\pmod{n_3}$. Meanwhile, similar to the above, it is easy to see that, (\ref{gsimu1}) and the case for $n_2$ in (\ref{gsimu}) still hold for the sequence $T^{3,l_3n_1n_2}$. Continuing this process, we will obtain a unique cyclic shift $T^{r,l_rn_1\cdots n_{r-1}}$ of $T$ such that (\ref{gsimu1}) and all cases in (\ref{gsimu}) hold simultaneously. Therefore, $T^{r,l_rn_1\cdots n_{r-1}}$ is a zero-sum sequence over $G$ of length $m$ which corresponds to $(n,m)$-Dyck path $P$. This completes the bijection between $\mathsf M(G,m)$ and $\mathcal{D}_{n,m}$.

For the bijection between $\mathsf M(G,m)$ and $\mathsf M(H,n)$, we use a necklace interpretation. Let $\mathfrak L$ be a necklace of $n+m$ beads with $n$ red beads and $m$ blue beads. Then proceeding in the clockwise direction from a fixed bead $\mathcal D$ (either red or blue), let $z_i$ ($i\in [0,n-1]$) be the number of blue beads between two successive red beads. Then by Remark \ref{labelG}, $(z_0,\ldots,z_{n-1})$ corresponds to a sequence over $G$ of length $m$. By the above proof, there exists a unique cyclic shift $T$ of $(z_0,\ldots,z_{n-1})$ such that $T$ corresponds to a zero-sum sequence over $G$. For example, in Figure \ref{fig1}, the necklace corresponds to the following 7 sequences over $C_7$: $(1,0,2,0,0,1,1)$, $(0,2,0,0,1,1,1),\ldots,(1,1,0,2,0,0,1)$. Moreover, $(0,0,1,1,1,0,2)$ is the only zero-sum sequence among these sequences.
\begin{figure}[ht]
\centering
\includegraphics[scale=0.25]{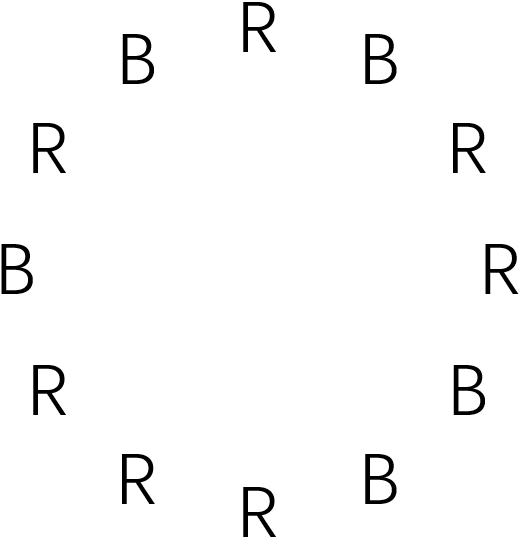}
\caption{A necklace with 7 red (R) beads and 5 blue (B) beads.}
\label{fig1}
\end{figure}

Therefore we have a one-to-one correspondence between all distinct necklaces of $n+m$ beads (with $n$ red beads and $m$ blue beads) and all zero-sum sequences over $G$ of length $m$. Let $S=(x_0,\ldots,x_{n-1})$ be a zero-sum sequence over $G$ of length $m$. Then $S$ corresponds to a necklace $\mathfrak L_S$ of $n+m$ beads with $n$ red beads and $m$ blue beads. Now, proceeding in the clockwise direction around $\mathfrak L_S$ from a fixed bead $\mathcal D$  (either red or blue), let $y_i$ ($i\in[0,m-1]$) be the number of red beads between two successive blue beads. Then we have a vector $T=(y_0,\ldots,y_{m-1})$ with $\sum_{i=0}^{n-1}y_i=n$. By Remark \ref{labelG}, $T$ corresponds to a sequence over $H$ of length $n$, though we do not necessarily have $\sigma(T)=0$. Similar to the above, we can find a unique cyclic shift $T'$ of $(y_0,\ldots,y_{m-1})$ such that $T'$ corresponds to a zero-sum sequence over $H$. For example, in above Figure \ref{fig1}, the necklace corresponds to the following 5 sequences over $C_5$: $(2,0,3,1,1)$, $(0,3,1,1,2),\ldots,(1,2,0,3,1)$. Moreover, $(1,2,0,3,1)$ is the only zero-sum sequence among these sequences. The converse correspondence is similar. This completes the proof of (1).

\medskip

(2) Let $G=C_n^r$ and $H=C_m^r$. Then for $d|n$ we have
$$\varphi_{G}(d)=\sum_{l|d}\mu(\frac{d}{l})l^r,$$
and for $d|m$ we have
$$\varphi_{H}(d)=\sum_{l|d}\mu(\frac{d}{l})l^r.$$
As $(n,m^r)=(n^r,m)$, for any $m\ge 1$ we have
\begin{align*}
\mathsf |M(G,m^r)|&=\frac{1}{n^r+m^r}\sum_{d|(n,m^r)}\varphi_{G}(d)\binom{\frac{n^r+m^r}{d}}{\frac{m^r}{d}}\\
&=\frac{1}{n^r+m^r}\sum_{d|(n^r,m)}\varphi_{H}(d)\binom{\frac{n^r+m^r}{d}}{\frac{n^r}{d}}=\mathsf |M(H,n^r)|.
\end{align*}

\medskip

(3) We assume that $|G|=n$. Let $S$ be a zero-sum subset of $G$ of cardinality $k$. Similarly, $S$ corresponds to a vector $(x_0,\ldots,x_{n-1})$ with $x_i\in\{0,1\}$ and $\sum_{i=0}^nx_i=k$. We shall construct a $(k,n-k)$-Dyck path corresponding to $S$. As $x_i\in\{0,1\}$ for any $i\in[0,n-1]$, we associate $S=(x_0,\ldots,x_{n-1})$ with a lattice path $L_S$ (not necessarily a Dyck path) as follows. Let 0 (resp. 1) represent a unit step $(0,1)$ (resp. $(1,0)$), then $S=(x_0,\ldots,x_{n-1})$ corresponds to a lattice path $L_S$ from $(0,0)$ to $(k,n-k)$ which only uses unit steps $(0,1)$ and $(1,0)$. Similar to the previous method, it is easy to see that there is a unique $(k,n-k)$-Dyck path among the cyclic shifts of $(x_0,\ldots,x_{n-1})$.

Conversely, let $P=(x_0,\ldots,x_{n-1})$ be a $(k,n-k)$-Dyck path with
\[
 x_i= \left\{\begin{array}{ll}&0,\quad \mbox{ if the $(i+1)$-th step uses $(0,1)$,}\\&1,\quad \mbox{ if the $(i+1)$-th step uses $(1,0)$},
\end{array} \right.
\]
where $i\in[0,n-1]$. Then by our previous construction, $(x_0,\ldots,x_{n-1})$ corresponds to a subset of $G$ of cardinality $k$. As $(k,n)=1$, similar to the previous method, it is easy to see that there is a unique cyclic shift $S$ of $(x_0,\ldots,x_{n-1})$ such that $S$ is a zero-sum subset of $G$. This completes the bijection between $\mathsf N(G,k)$ and $\mathcal{D}_{k,n-k}$.

Now we construct the bijection between $\mathsf N(G,k)$ and $\mathsf N(G,n-k)$. Let $S=(x_0,\ldots,x_{n-1})$ be a zero-sum subset of $G$ of cardinality $k$. Let $T=(y_0,\ldots,y_{n-1})$ where $y_i=1-x_i$ for any $i\in[0,n-1]$. Then it is clear that $T$ corresponds to a subset of $G$ of cardinality $n-k$. Similar to the above method, we can find a unique cyclic shift $T'$ of $T$ such that $T'$ corresponds to a zero-sum subset over $G$ of length $n-k$. The converse correspondence is similar. This completes the proof.
\qed

\begin{remark}\label{Johnsonproof}
For positive integers $n,m$ with $(n,m)=1$, in \cite{And}, Anderson provided a well-known and very elegant bijection between the set of $(n,m)$-core partitions and $\mathcal{D}_{n,m}$ using the abacus construction. Later in \cite{John}, while proving a conjecture of Armstrong \cite{AHJ} concerning the average size of the $(n,m)$-core partitions, as a byproduct, Johnson obtained an interesting bijection between $\mathsf M(C_n,m)$ and the set of $(n,m)$-core partitions. He employed the abacus construction and various coordinate changes from the perspective of the Ehrhart theory. Therefore, combining these results, although slightly complicated, one obtains a bijection between $\mathsf M(C_n,m)$ and $\mathcal{D}_{n,m}$. Our proof of Theorem \ref{mainth1}.(1) provides a direct bijection between $\mathsf M(G,m)$ and $\mathcal{D}_{n,m}$ for any finite abelian group $G$ of order $n$, without using the abacus construction.
\end{remark}

Next, we turn to prove Theorem \ref{subset-reci}. Firstly, it is easy to see that if $G$ satisfies $\sum_{g\in G}g=0$, then we have $|\mathsf N(G,k)|=|\mathsf N(G,|G|-k)|$ holds for any $k\in[1,|G|-1]$. We have the following lemma.

\begin{lemma}\label{allsum0}
Let $G$ be a finite abelian group. We have $\sum_{g\in G}g=0$ if and only if $|G|$ is odd or $G\cong C_{n_1}\oplus\cdots\oplus C_{n_r}$ with $r\ge2$ and $2|n_{r-1}$.
\end{lemma}

\begin{proof} If $|G|$ is odd, then obviously we have $\sum_{g\in G}g=0$. If $G\cong C_{n_1}\oplus\cdots\oplus C_{n_r}$ with $r\ge2$ and $2|n_{r-1}$, then all elements of order 2 in $G$ forms a subgroup $H$ and $H\cong C_2^m$ with $2\le m\le r$. Consequently, we have $\sum_{g\in G}g=\sum_{h\in H}h=0$. Conversely, we assume that $\sum_{g\in G}g=0$. If $G$ does not satisfy any one of the desired properties, then we must have $G\cong C_{n_1}\oplus\cdots\oplus C_{n_{r-1}}\oplus C_{n_r}$ with $r\ge1$, $2|n_r$, and $(n_1\cdots n_{r-1},2)=1$. It follows that $G$ has only one element of order 2, which we denote by $e$. Consequently, we have $\sum_{g\in G}g=e\neq 0$, which is a contradiction.
\end{proof}

Next, we provide a sufficient condition for $|\mathsf N(G,k)|=|\mathsf N(G,|G|-k)|$ holds when $\sum_{g\in G}g\neq0$. Recall that for any positive integer $m$, $v_2(m)$ is defined as the 2-adic valuation of $m$.

\begin{lemma}\label{v2kv2n}
Let $G\cong C_{n_1}\oplus\cdots\oplus C_{n_r}$ be a finite abelian group and $k$ a positive integer with $k\le |G|-1$. If $v_2(k)<v_2(n_r)$, then we have $|\mathsf N(G,k)|=|\mathsf N(G,|G|-k)|$.
\end{lemma}

\begin{proof}
By Lemma \ref{allsum0}, we may assume that $G\cong C_{n_1}\oplus\cdots\oplus C_{n_{r-1}}\oplus C_{n_r}$ with $r\ge1$, $2|n_r$, and $(n_1\cdots n_{r-1},2)=1$. Let $e$ be the unique element of order 2 in $G$. Let $S$ be a zero-sum subset of $G$ of length $k$. We claim that $e\in kG=\{kg\ |\ g\in G\}=dG$, where $d=(k,n_r)$. Let $n_r=n_r'd$ and $k=k'd$. Let $u\in G$ such that ord$(u)=n_r$. Since $v_2(k)<v_2(n_r)$, we have $\frac{n_r'}{2}\in\mathbb N$ and consequently  $e=\frac{n_r}{2}u=d\frac{n_r'}{2}u\in dG=kG$. Therefore, there exists $x\in G$ such that $e=kx$. Obviously, we have $\sum_{g\in G\setminus S}g=e$. If $G\setminus S=\{g_1,\ldots,g_{n-k}\}$, we define $x+G\setminus S=\{x+g_1,\ldots,x+g_{n-k}\}$. It is easy to see that $\sum_{g\in x+G\setminus S}=-kx+e=0$, which means that $x+G\setminus S$ is a zero-sum subset of $G$ of length $|G|-k$. Note that the above construction provides a bijection between $\mathsf N(G,k)$ and $\mathsf N(G,|G|-k)$ in this case.
\end{proof}

Now, we are ready to prove Theorem \ref{subset-reci}.

\medskip

{\sl Proof of the Theorem \ref{subset-reci}.} If one of the conditions (1)-(3) holds, by Lemmas \ref{allsum0} and \ref{v2kv2n}, we have $|\mathsf N(G,k)|=|\mathsf N(G,|G|-k)|$. Conversely, if the equality $|\mathsf N(G,k)|=|\mathsf N(G,|G|-k)|$ holds, we have to prove that at least one of the conditions (1)-(3) holds. If (1) or (2) holds, then we are done. So we may assume that both (1) and (2) fail. It follows that
$G\cong C_{n_1}\oplus\cdots\oplus C_{n_{r-1}}\oplus C_{n_r}$ with $r\ge1$, $2|n_r$, and $(n_1\cdots n_{r-1},2)=1$. In this case, we have to show that (3) holds, that is $v_2(k)<v_2(n_r)$. We assume to the contrary that $v_2(k)\ge v_2(n_r)$.

By Theorem \ref{LiW}, it is easy to see that
$$|\mathsf N(G,k)|=\frac{1}{n}\sum_{d|(n,k)}\varphi_G(d)(-1)^{k+\frac{k}{d}}\binom{n/d}{k/d}$$
and
\begin{align*}
|\mathsf N(G,n-k)|&=\frac{1}{n}\sum_{d|(n,n-k)}\varphi_G(d)
(-1)^{n-k+\frac{n-k}{d}}\binom{n/d}{(n-k)/d}\\
&=\frac{1}{n}\sum_{d|(n,k)}\varphi_G(d)
(-1)^{n-k+\frac{n-k}{d}}\binom{n/d}{k/d}.
\end{align*}
Therefore, the only possible difference of these two formulas is the sign of each summand, that is $(-1)^{k+\frac{k}{d}}$ and $(-1)^{n-k+\frac{n-k}{d}}$. Note that, under our assumption that $v_2(k)\ge v_2(n_r)\ge 1$, we have $(-1)^{k+\frac{k}{d}}=(-1)^{\frac{k}{d}}$ and $(-1)^{n-k+\frac{n-k}{d}}=(-1)^{\frac{n-k}{d}}$.

We distinguish two cases.

{\bf Case 1.} $v_2(k)=v_2(n_r)=t$. It is easy to see that, for any $d|(n,k)$, $\frac{n-k}{d}$ is even, which means that all summands in $|\mathsf N(G,|G|-k)|$ are positive. However, when $d=2^t$, $\frac{k}{2^t}$ is odd, which means that at least one summand in $|\mathsf N(G,k)|$ is negative. It follows that $|\mathsf N(G,k)|<|\mathsf N(G,|G|-k)|$, which is a contradiction.

{\bf Case 2.} $v_2(k)>v_2(n_r)=t$. Similarly, it is easy to see that, for any $d|(n,k)$,  $\frac{k}{d}$ is even, which means that all summands in $|\mathsf N(G,k)|$ are positive. However, when $d=2^t$, $\frac{n-k}{2^t}$ is odd, which means that at least one summand in $|\mathsf N(G,|G|-k)|$ is negative. It follows that $|\mathsf N(G,k)|>|\mathsf N(G,|G|-k)|$, which is a contradiction.
\qed

Next, we are going to prove Theorem \ref{mainth2}. We recall (\ref{abel0}) in the following:
$$\mathcal{F}\Big(\big(\mathcal{S}(\mathcal{R})\otimes\wedge(\mathcal{R})\big)_{G};s,t\Big)
=\frac{1}{|G|}\sum_{d\ge1}\Big(\varphi_G(d)\Big(\frac{1-(-t)^d}{1-s^d}\Big)^{|G|/d}\Big).$$

{\sl Proof of the Theorem \ref{mainth2}.} By (\ref{abel0}) and after extracting the coefficient of $s^pt^m$, we obtain that
\begin{align*}
&\dim\big(\mathcal{S}^p(\mathcal{R})\otimes\wedge^m(\mathcal{R})\big)_{G_{q+m}}\\
&=\frac{1}{p+q+m}\sum_{d|(p,q,m)}(-1)^{m+\frac{m}{d}}\varphi_{G_{q+m}}(d)\binom{(p+q+m)/d}{p/d,q/d,m/d}
\end{align*}
and
\begin{align*}
&\dim\big(\mathcal{S}^q(\tilde{\mathcal{R}})\otimes\wedge^m(\tilde{\mathcal{R}})\big)_{H_{p+m}}\\
&=\frac{1}{p+q+m}\sum_{d|(p,q,m)}(-1)^{m+\frac{m}{d}}\varphi_{H_{p+m}}(d)\binom{(p+q+m)/d}{p/d,q/d,m/d}.
\end{align*}
Under the assumption that $(p,q,m)=1$, the desired result follows immediately.\qed

\begin{remark} Panyushev \cite{Pan} asked for a combinatorial interpretation of (\ref{Panyushev}). Here we will provide one under a further assumption that $(p,q+m)=(q,p+m)=1$. In this case, by the definition of the symmetric tensor exterior algebra, it is easy to see that calculating
$$\dim\big(\mathcal{S}^p(\mathcal{R})\otimes\wedge^m(\mathcal{R})\big)_{G_{q+m}}$$ is equivalent to counting the number of pairs
\begin{equation}\label{AB}
(A,B),
\end{equation}
where $A$ is a sequence over $G_{q+m}$ of length $p$ and $B$ is a subset of $G_{q+m}$ of cardinality $m$ such that $\sigma(A)+\sigma(B)=0$ in $G_{q+m}$. Note that
$$\frac{1}{p+q+m}\binom{p+q+m}{p,q,m}=\frac{1}{p+q+m}\binom{p+q+m}{p}\binom{q+m}{m}.$$
Suppose that $(S,T)$ is a pair satisfying the above condition (\ref{AB}). In order to provide the bijection, we have to assign $(S,T)$ to a unique pair $(U,V)$ where $U$ is a sequence over $H_{p+m}$ of length $q$ and $V$ is a subset of $H_{p+m}$ of length $m$, such that $\sigma(U)+\sigma(V)=0$ in $H_{p+m}$.

Firstly, we construct an uncolored necklace $\mathfrak{L}$ of $p+q+m$ beads with a fixed bead which we denote by $\mathfrak{D}$ (as starting point). Based on Remark \ref{important}, we may assume that $S=(x_0,\ldots,x_{q+m-1})$ and $T=(y_0,\ldots,y_{q+m-1})$. Now, based on $(S,T)$, we color $\mathfrak{L}$ with three colors: red, green, and blue to obtain a new necklace $\mathfrak{L}_{S,T}$. The resulting $\mathfrak{L}_{S,T}$ is a necklace with $p+q+m$ beads ($p$ red beads, $m$ green beads, and $q$ blue beads) and these beads are located as follows. Proceeding in the clockwise direction from $\mathfrak{D}$, let
$$x_0,\ldots,x_{q+m-1}$$
be the number of red beads between two successive green or blue beads. Also proceeding in the clockwise direction from $\mathfrak{D}$, among these $q+m$ green or blue beads, the $i$th bead is green if $y_{i-1}=1$ and blue otherwise for $0\le i\le q+m-1$.

Proceeding in the clockwise direction from $\mathfrak{D}$, we define
$$W=(u_0,\ldots,u_{p+m-1}),$$
where the $u_i$'s are the number of blue beads between two successive red or green beads. Also proceeding in the clockwise direction from $\mathfrak{D}$, we define
$$V=(v_0,\ldots,v_{p+m-1}),$$
where
$v_{i-1}=1$ if the $i$-th bead is green among those $p+m$ red or green beads and 0 otherwise, for $1\le i\le p+m$. Then $V$ corresponds to a subset of $H_{p+m}$ of cardinality $m$. Consequently, similar to the above proofs, there is a unique cyclic shift $W^l$ of $W$ such that $\sigma(W^l)+\sigma(V)=0$. Let $U=W^l$, then $(U,V)$ is exactly what we want. As we have a fixed bead $\mathfrak{D}$ and we only allow the corresponding sequences (the number of blue beads between two successive red or green beads) to be shifted (the relative positions of red and green beads have not changed), it is easy to see that, in this way, $(U,V)$ is the unique pair corresponding to $(S,T)$. The converse correspondence is similar. This completes the bijection.
\end{remark}

\section{Concluding remarks}

Recently, the interplay of zero-sum theory with invariant theory and with factorization theory (see \cite{Gero,GeroZho}) has become a very interesting and active research field. So far, in most cases, techniques and useful concepts of zero-sum theory (or more generally, arithmetic combinatorics) have been very successfully used in the study of invariant theory and factorization theory (see \cite{CDG} for a very nice exposition). While, powerful techniques and ideas in invariant theory and factorization theory are less often used in the study of zero-sum theory (or arithmetic combinatorics). We mention some recent studies in this direction. Fan and Tringali \cite{FanSalvo} provided a very good example of using ideas from factorization theory to study arithmetic combinatorics. A very recent paper of Bashir, Geroldinger and Zhong \cite{BaGeZh} studied an interesting zero-sum problem arising from factorization theory. Domokos \cite{Domo} provided a new and very interesting perspective to look at zero-sum theory and factorization theory from a invariant theoretic point of view.

The study of this paper was partially motivated by some studies related to the Poincar\'e series (\ref{Almkvist}), which is an important and very useful tool in invariant theory. Besides leading to a different proof of Theorem \ref{LiW}, the idea used in the proof of Lemma \ref{Panlemma} also has the potential to be used to deal with some other similar counting problems in arithmetic combinatorics. We hope to further employ methods and ideas from invariant theory and factorization theory to discover more and more interesting and useful results in zero-sum theory and arithmetic combinatorics.

Rational Catalan combinatorics studies rational generalizations of combinatorial objects which are counted by the Catalan number $\mathsf{Cat}_n$. It is known \cite{Stan1} that there are at least 200 distinct families of combinatorial objects counted by $\mathsf{Cat}_n$. While, so far, the number of families of combinatorial objects counted by $\mathsf {Cat}_{n,m}$ is much less. Note that, $\mathsf {Cat}_{n,m}$ is not just a simple generalization of the classical one. In some cases, the symmetry of $n$ and $m$ in $\mathsf {Cat}_{n,m}$ leads to some surprising and unexpected symmetry or duality, we refer to, e.g., \cite{ARW} for a discussion. Our study (Theorem \ref{mainth1}) provides a new interpretation for the rational Catalan numbers, which indicates that there exists a special symmetry (i.e., $|\mathsf M(G,|H|)|=|\mathsf M(H,|G|)|$) between any two finite abelian groups of co-prime cardinality (irrelevant to their group structures). Our bijective proof, passing through the rational Dyck paths, provides an intuitive interpretation of this symmetry. The following result shows that the reciprocity $|\mathsf M(G,|H|)|=|\mathsf M(H,|G|)|$ does not always hold for two general finite abelian groups $G$ and $H$.

\begin{proposition}\label{GCp}
Let $G\cong C_{n_1}\oplus\cdots\oplus C_{n_r}$ be a finite abelian group of order $n$. Let $p$ be a prime and $C_p$ the cyclic group of order $p$. Then we have $|\mathsf M(G,|C_p|)|=|\mathsf M(C_p,|G|)|$ if and only if $(n_1\cdots n_{r-1},p)=1$.
\end{proposition}

\begin{proof}
By Theorem \ref{LiW}, we have
$$|\mathsf M(G,p)|=\frac{1}{n+p}\sum_{d|(n,p)}\varphi_G(d)\binom{n/d+p/d}{n/d}$$
and
$$|\mathsf M(C_p,n)|=\frac{1}{n+p}\sum_{d|(n,p)}\varphi_{C_p}(d)\binom{n/d+p/d}{n/d}$$
where $\varphi_G(d)=\sum_{l|d}\mu(d/l)\prod_{i=1}^r(n_i,l)$ and $\varphi_{C_p}(d)=\sum_{l|d}\mu(d/l)(p,l)$. We first assume that $(n_1\cdots n_{r-1},p)=1$. By Theorem \ref{mainth1}, it suffices to consider the case when $p|n_r$. In this case, we have
$$|\mathsf M(G,p)|=\frac{1}{n+p}(\binom{n+p}{n}+(p-1)\binom{n/p+1}{n/p})$$
and
$$|\mathsf M(C_p,n)|=\frac{1}{n+p}(\binom{n+p}{n}+(p-1)\binom{n/p+1}{n/p}).$$
Therefore we have $|\mathsf M(G,|C_p|)|=|\mathsf M(C_p,|G|)|$.

Conversely, we assume that $|\mathsf M(G,|C_p|)|=|\mathsf M(C_p,|G|)|$. If $(n_1\cdots n_{r-1},p)=p$, then clearly we have $\varphi_G(p)\ge p^2-1>p-1=\varphi_{C_p}(p)$. Then we have
$$|\mathsf M(G,p)|=\frac{1}{n+p}(\binom{n+p}{n}
+\varphi_G(p)\binom{n/p+1}{n/p})$$
and
$$|\mathsf M(C_p,n)|=\frac{1}{n+p}(\binom{n+p}{n}+\varphi_{C_p}(p)\binom{n/p+1}{n/p}).$$
Consequently we have $|\mathsf M(G,|C_p|)|>|\mathsf M(C_p,|G|)|$, which is a contradiction.
\end{proof}

Based on Theorems \ref{mainth1} and \ref{subset-reci} and Proposition \ref{GCp}, it is natural to propose the following problem for future study.

\begin{problem}
Let $G$ and $H$ be finite abelian groups. What are the necessary and sufficient conditions (on $G$ and $H$) for $|\mathsf M(G,|H|)|=|\mathsf M(H,|G|)|$ holds?
\end{problem}

\subsection*{Acknowledgments}

We are deeply indebted to the referee who provided us with many helpful comments. In particular, Theorem \ref{subset-reci} is motivated by an insightful comment of the referee. A part of this work was started during a visit by the second author to the Karl-Franzens-Universit\"{a}t Graz in the spring semester 2019, he would like to thank Prof. Alfred Geroldinger for invitation and providing him with a wonderful working environment, as well as many helpful comments on the manuscript. We thank Dr. Qinghai Zhong for helpful discussions. We are also grateful to our advisor Prof. Weidong Gao for his support and encouragement all the time. D.C. Han was supported by the National Science Foundation of China Grant No.11601448. H.B. Zhang was supported by the National Science Foundation of China Grant No.11901563.

\bibliographystyle{amsplain}

\end{document}